\documentclass{amsart}
\usepackage[utf8]{inputenc}
\usepackage{graphicx,hyperref}
\usepackage{amsthm}
\usepackage{amssymb}
\usepackage{latexsym}
\usepackage{amsmath, amsthm}
\usepackage{float, xcolor}
\usepackage{fullpage}

\theoremstyle{plain}
\newtheorem{theorem}{Theorem}[section]
\newtheorem{proposition}[theorem]{Proposition}
\newtheorem{lemma}[theorem]{Lemma}
\newtheorem{corollary}[theorem]{Corollary}

\theoremstyle{definition}
\newtheorem{definition}[theorem]{Definition}
\newtheorem{example}[theorem]{Example}
\newtheorem{remark}[theorem]{Remark}

\newcommand{\supp}{\operatorname{supp}}

\newcommand{\col}{\operatorname{col}}
\newcommand{\genus}{\operatorname{genus}}
\newcommand{\sgn}{\operatorname{sgn}}
\newcommand{\norm}[1]{\left\lVert#1\right\rVert}
\newcommand{\abs}[1]{\left\lvert#1\right\rvert}
\newcommand{\inner}[1]{\left\langle#1\right\rangle}
\newcommand{\Row}{\operatorname{Row}}
\newcommand{\codim}{\operatorname{codim}}
\newcommand{\rk}{\operatorname{rk}}
\newcommand{\calF}{\mathcal{F}}
\newcommand{\calI}{\mathcal{I}}
\newcommand{\calM}{\mathcal{M}}
\newcommand{\R}{\mathbb{R}}
\newcommand{\calC}{\mathcal{C}}
\newcommand{\Z}{\mathbb{Z}}

\begin{document}
\title{Lattice of Integer Flows and the Poset of Strongly Connected Orientations for Regular Matroids}
\author{Zsuzsanna Dancso} 
\address{University of Sydney}
\email{\href{mailto:zsuzsanna.dancso@sydney.edu.au}{zsuzsanna.dancso@sydney.edu.au}}
\urladdr{\url{http://zsuzsannadancso.net}}

\author{Jongmin Lim}
\email{\href{mailto:jongjum18@naver.com}{jongjum18@naver.com}}

\def\subjclassname{\textup{2020} Mathematics Subject Classification}
\expandafter\let\csname subjclassname@1991\endcsname=\subjclassname
\expandafter\let\csname subjclassname@2000\endcsname=\subjclassname
\subjclass{
    05C21, 05C50
    \hfill
    Date: \today
}

\maketitle

\begin{abstract} 
A 2010 result of Amini provides a way to extract information about the structure of the graph from the geometry of the Voronoi polytope of the lattice of integer flows (which determines the graph up to two-isomorphism). Specifically, Amini shows that the face poset of the Voronoi polytope is isomorphic to the poset of {\em strongly connected orientations} of subgraphs. This answers a question raised by Caporaso and Viviani, and Amini also proves a dual result for integer cuts. In this paper we generalise Amini's result to regular matroids; in this context the theorem for integer cuts becomes a direct consequence of the theorem for integer flows, by making duality explicit as matroid duality.
\end{abstract}

\tableofcontents

\section{Introduction}
Given a finite, connected graph $G$ (multiple edges allowed), the {lattices of integer flows and cuts} are algebraic graph invariants associated to $G$. Here, a {\em lattice} means a finitely generated free abelian group with a symmetric, non-degenerate $\Z$-valued inner product. 
In their seminal paper \cite{BHN} Bacher, de la Harpe, and Nagnibeda proved that these lattices are {\em two-isomorphism} invariants, that is, they are determined by the matroid isomorphism class of the graphical matroid -- or cycle matroid -- associated to $G$. They remarked that they were unable to find a pair of non-isomorphic 3-connected graphs with isomorphic lattices of integer flows. 

This lead to the question: exactly how much information do these lattices retain about the two-isomorphism class of a graph? Lattices of integer flows and cuts extend naturally to the context of regular matroids, where the corresponding question is whether they determine the isomorphism class of regular matroids. The same question also arises in tropical geometry, where it is known as the tropical Torelli problem.

In 2010, these questions were independently settled by Caporaso--Viviani \cite{CV} for graphs and tropical curves, and and Su--Wagner \cite{Su-Wagner} for regular matroids. In fact, the graph case also follows from earlier equivalent results of Watkins on graph Laplacians \cite{W1, W2}. After their geometric analogues, these results are also collectively referred to as Discrete Torelli theorems.

Recall that the class of regular matroids contains graphical matroids, and is closed under matroid duality. The main theorem of \cite{Su-Wagner} states that the lattice of integer flows determines the isomorphism class of a regular matroid up to co-loops. For graphs, this means that the lattice of integer flows uniquely determines the two-isomorphism class of {\em two-edge-connected} graphs. A dual result states that the lattice of integer cuts determines the isomorphism class of a regular matroid up to loops. Dually, the lattice of integer cuts determines the two-isomorphism class of graphs with no loop edges.

Amini's Theorem \cite{Amini} -- which settles Conjecture 5.2.8 of \cite{CV} -- extracts information encoded in the lattice of integer flows from the geometry of the {\em Voronoi cell} of the lattice. Namely, Amini shows that the face poset of the Voronoi cell is isomorphic to the poset of {\em strongly connected orientations} of subgraphs of the graph.
Amini also proves a dual theorem for the lattice of integer cuts in \cite{Amini}. This dual theorem later appeared in \cite{AE20} in the context of toric geometry, as it is crucial in the study of the degeneration problem for linear series on curves \cite{AE20-2,AE21}. Note that the framework in \cite{AE20} is more general: the result required for the geometric applications is a mixed setting where the tiling is given by a collection of Voroni polytopes associated to some subgraphs of $G$, which are determined by the arithmetic of the integer edge lengths, and the divisor. In this paper we build on the original setting of \cite{Amini} with particular focus on the duality between the flow and cut theorems.

For a planar graph $G$, the lattice of integer cuts of $G$ coincides with the lattice of integer flows of the planar dual $G^*$. This statement does not extend to non-planar graphs, as there is no notion of duality. However, the lattices of integer cuts and flows are invariants of the {\em graphical matroid} associated to a graph, and graphical matroids embed in the larger class of {\em regular matroids}, which is closed under matroid duality. The matroid dual of a graphical matroid is called a {\em co-graph}, and the intersection of graphical matroids and co-graphs are exactly the graphical matroids of planar graphs. In this more general context, the lattice of integer cuts of a regular matroid is isomorphic to the lattice of integer flows of the dual matroid. 

In this note, we generalise Amini's theorem to regular matroids: the proof transfers to the new context in a fairly natural way with appropriate language in place.
This generalisation eliminates the need to supply a separate proofs for the two dual theorems. We show how the theorem for integer cuts -- relating the geometry of the Voronoi cell of the lattice of {\em integer cuts} to {\em cut subgraphs} with {\em coherent acyclic} orientations -- follows directly, via matroid duality, from the theorem for integer flows. We note that the generalisation to regular matroids does not rely on Seymour's decomposition theorem but rather uses elementary properties of totally unimodular matrices.

\subsection*{Acknowledgements} We thank Omid Amini for helpful comments, and the anonymous referee for their careful report and editing suggestions. This work was carried out as part of the University of Sydney's Talented Student Program as an undergraduate project by JL, supervised by ZD.  ZD was partially supported by an ARC DECRA DE170101128 award.

\section{Amini's Theorem}

We begin with a brief outline of Amini's Theorem \cite{Amini}  for graphs, recalling basic definitions and aiming the presentation towards the generalisation to regular matroids. 

Let $G(V, E)$ be an undirected graph with possible multiple edges and loops. Choose an arbitrary orientation for $G$, that is, a direction for each edge. We denote a directed  edge $e$ beginning at $u \in V(G)$ and ending at $v\in V(G)$ by $e=(u,v)$. Let $M$ be the signed incidence matrix of this directed graph: the columns of $M$ are indexed by edges, the rows are indexed by vertices, and for a vertex $v \in V(G)$ and edge $e \in E(G)$, the entry $M_{ve}$ is given by
\begin{align*}
M_{ve} = \begin{cases}
-1 & \text{if $e$ points away from $v$,}\\
+1 & \text{if $e$ points towards $v$,}\\
\phantom{-}0 & \text{if $e$ is a loop, or $e$ is not incident to $v$.}
\end{cases}
\end{align*}
Let $\mathbb{R}^E$ and $\mathbb{R}^V$ be real vector spaces with standard bases indexed by $E=E(G)$ and $V=V(G)$ respectively. The matrix $M$ represents a linear transformation $M:\mathbb{R}^E \to \mathbb{R}^V$ such that for a directed edge $e=(u,v)$, $M(e) = v-u$. The {\em vector space of real valued flows} $\calF(G)$, also known as the {\em cycle space} of $G$, is defined as $\mathcal{F}(G):=\ker(M)$. The Euclidean inner product $\inner{ -, - }:\mathbb{R}^E \times \mathbb{R}^E \to \mathbb{R}$ and the corresponding norm $\norm{\cdot}: \mathbb{R}^E \to \mathbb{R}_{\geq 0}$ are defined by:
\begin{align*}
\inner{ e_1,  e_2} &= \begin{cases}
1 &  e_1 =  e_2\\
0 &  e_1 \neq  e_2
\end{cases} & & \text{for any }  e_1,  e_2\in E(G),\\
\norm{{x}}^2 &= \langle  x,  x\rangle & & \text{for any } x\in \mathbb{R}^E.
\end{align*}
An {\em integer lattice}, in general, is a finitely generated free abelian group with a symmetric non-degenerate $\mathbb{Z}$-valued bilinear form. The {\em lattice of integer flows} $\Lambda(G)$ of the graph $G$ is the integer lattice given by $\Lambda(G):= \ker(M) \cap \mathbb{Z}^E$. The spaces $\mathcal{F}(G)$ and $\Lambda(G)$ inherit the bilinear and quadratic forms of $\mathbb{R}^E$, this equips $\Lambda(G)$ with a lattice structure. The {\em genus of the graph} $G$ is $\genus(G) := \dim(\mathcal{F}(G))= \rk (\Lambda(G))$.\\
The isomorphism (isometry) classes of $\mathcal{F}(G)$ and $\Lambda(G)$ as inner product space and integral lattice, respectively, are independent of the orientation chosen on $G$: reversing the orientation of an edge $e$ merely reflects the cycle space with respect to the coordinate hyperplane orthogonal to the $e$-axis.

\begin{definition}
Let $ \lambda \in \Lambda(G)$. The Voronoi cell of $ \lambda$ is $V_\lambda := \{  x \in \mathcal{F}(G) \ | \ \norm{ x -  \lambda} \leq \norm{ x -  \mu},  \forall  \mu \in \Lambda(G)\}$.
\end{definition}

In other words, the Voronoi cell of a lattice point $\lambda$ is the set of points in the space $\mathcal{F}(G)$ that are nearest to $\lambda$ of all the lattice points in $\Lambda(G)$. The Voronoi cells are polyhedra which tessellate $\mathcal{F}(G)$, and they are invariant under translation by lattice vectors, as the lattice itself is invariant under such translations. Thus, we only consider the Voronoi cell of the origin, $V_0$. 

\begin{definition}
Let $\mathcal{FP}(G)$ denote the {\em poset of faces} of $V_0$ ordered by inclusion. The face poset is {\em graded} by dimension: that is, same-dimensional faces are never comparable, and when two faces are comparable, the higher dimensional one contains the lower dimensional one.
\end{definition}

\begin{definition}
An orientation of a graph $G$ is {\em strongly connected} if for every directed edge $e=(u,v)$, there exist a directed path from $v$ to $u$. 
\end{definition}

Note that a graph $G$ with a strongly connected orientation need not be connected. However, if there exists a strongly connected orientation for $G$, then the connected components of $G$ must be {\em 2-edge-connected}, that is, remain connected upon removal of any single edge. 

\begin{definition} \cite{Amini} The {\em poset of strongly connected orientations}, denoted
$\mathcal{SC}(G)$, consists of all ordered pairs of the form $(H, D_H)$ where $H$ is a (not necessarily connected) subgraph of $G$, and $D_H$ is a strongly connected orientation on $H$. A subgraph here means a subset of the edges on the full set of vertices. 
The partial ordering on $\mathcal{SC}(G)$ is defined as follows: $(H, D_H) \leq (H', D_{H'})$ if and only if $H'$ is a subgraph of $H$, and $D_{H'}$ is the induced orientation on $H'$ by $D_H$. The maximal element of this poset, by definition, is the empty subgraph with the empty orientation $(\emptyset, \emptyset)$.
\end{definition}

The following fact is stated in \cite{Amini}:
\begin{proposition}\label{prop:PosetGraded} 
The poset $\mathcal{SC}(G)$ is graded: the degree of $(H,D_H)$ is defined as $\operatorname{genus}(G)-\operatorname{genus}(H)$, were the genus of a disconnected graph is understood to be the sum of the genera of its connected components.
\end{proposition}
We include a proof, which will apply directly also to regular matroid version of the same claim:
\begin{proof}
Let $H'\subsetneq H$ be subgraphs of $G$ with strongly connected orientations $D_H$ and $D_{H'}$, such that $(H, D_H) \leq (H', D_{H'})$. We need to prove that $\genus(H')<\genus(H)$.

Let $M_H$ and $M_{H'}$ be the signed incidence matrices of $H$ and $H'$ respectively. Since $D_H$ induces $D_{H'}$ on $H'$, we know that $M_{H'}$ is a submatrix of $M_H$, that is, $M_{H'}$ consists of some, but not not all, columns of $M_H$. Thus there is a canonical embedding $\ker M_{H'} \hookrightarrow \ker M_H$, and therefore  $\genus(H')\leq\genus(H)$.

To prove strict inequality, let $e \in E(H) \setminus E(H')$. As $H$ is strongly connected, if $e=(u,v)$, there exists a directed path $x$ from $v$ to $u$. Consider the sum of the edges of $x$, as an element of $\mathbb{R}^E$. By an abuse of notation, denote this element also by $x$. Since $e$ forms a directed cycle with $x$, we have $ e + x \in \ker M_H$. Furthermore, $ e+ x$ is not in the image of $\ker M_{H'} \hookrightarrow \ker M_H$, as $e \not\in H'$. Thus,  $\genus(H') = \dim\ker M_{H'} < \dim\ker M_H = \genus(H)$.
\end{proof}
The following theorem is a main result of \cite{Amini}, which establishes an intricate connection between the structure of the graph and the geometry of the lattice of integer flows:
\begin{theorem}
For a finite graph $G$, $\mathcal{FP}(G) \cong \mathcal{SC}(G)$ as graded posets.
\end{theorem}

\section{Generalisation to Regular Matroids}
Matroids are a combinatorial abstraction and generalisation of the properties of linear dependence in vector spaces, as well as cycles in graphs. Terminology borrows from both linear algebra and graph theory. Matroid theory has applications in a range of both pure and applied fields. In the context of this paper, the main advantage of working with matroids over graphs is the notion of {\em matroid duality}, which generalises planar graph duality.

\begin{definition}
A (finite) {\em matroid} $\mathcal{M}(E, \mathcal{I})$ is an ordered pair of a finite set $E$, called the {\em ground set}, and $\mathcal{I}$, the ``family of independent subsets of $E$'' which satisfies the following axioms
\begin{enumerate}
\item $\emptyset \in \mathcal{I}$.
\item If $A\in \calI$, then $A'\in\mathcal{I}$ for all $A' \subset A$.
\item If $A, B \in \mathcal{I}$ with $|A| > |B|$, then $\exists x \in A\setminus B$ such that $B\cup \{x\} \in \mathcal{I}$.
\end{enumerate}
\end{definition}

Maximal independent sets in a matroid  are called {\em bases}. A subset of $E$ which is not independent is {\em dependent}; dependent sets in matroids are also called {\em cycles}, and minimally dependent sets are called {\em circuits}. Matroids have several equivalent definitions in terms of bases, cycles, etc. We include the equivalent definition in terms of cycles, as a set-up for oriented matroids:

\begin{definition}
A (finite) {\em matroid} $\mathcal{M}(E, \mathcal{I})$ is an ordered pair of a finite set $E$, called the {\em ground set}, and $\mathcal{C}$, the ``family of circuit subsets of $E$'' which satisfies the following axioms
\begin{enumerate}
\item $\emptyset \notin \mathcal{C}$.
\item No circuit is properly contained in another circuit.
\item If $C_1, C_2 \in \mathcal{C}$ are distinct circuits with $e \in C_1\cap C_2$, then $C_1\cap C_2\setminus \{e\}$ contains a circuit.
\end{enumerate}

\end{definition}

\begin{example} The following are two motivating examples of matroids:
\begin{enumerate}
\item If $V$ is vector space over a field $\mathbb F$, any set (arrangement) $E$ of vectors of $V$ forms the ground set of a matroid whose independent subsets are the subsets of linearly independent vectors. Matroids arising this way are called {\em representable} over $\mathbb F$. An $\mathbb F$-representation of the matroid is a matrix $M$, whose columns are the vectors in $E$, expressed in some fixed basis, and $E$ -- with the linearly independent subsets -- is called the {\em column matroid} of $M$.

\item Given a finite graph $G$, let $E$ be the set of edges of the graph $G$, and $\mathcal{I}$ be the family of cycle-free subsets of $E$; these are called {\em forests}. The pair $\mathcal{M}(G)=(E, \mathcal{I})$ is a matroid, called the {\em graphical matroid} (or {\em cycle matroid}) associated to $G$. In the graphical matroid, cycles are graph cycles, circuits are minimal cycles, and bases are spanning forests.
Note that $\mathcal{M}(G)$ is representable, over any field, as the column vectors of the signed incidence matrix $M_G$, given any choice of orientation for $G$. 
\end{enumerate}
\end{example}

\begin{definition}
A matroid is {\em regular} if it is representable over any field $\mathbb{F}$.
\end{definition}

By a theorem of Tutte \cite[Theorem 6.6.3]{Oxley}, regular matroids are exactly those which can be represented as the column matroid of some  {\em totally unimodular} $r\times m$ matrix of rank $r$, over the field $\mathbb R$. The columns of $M$ are indexed by the elements of the ground set $E$. A matrix is {\em totally unimdoular}, or TU for short, if every square submatrix has determinant $-1$, $0$, or $1$. For example, a TU matrix representing a graphical matroid can be obtained from the signed incidence matrix of the graph by deleting any row. In other words, graphical matroids are always regular.

Two $\mathbb F$-representations $M$ and $M'$ of a matroid are {\em equivalent} if there is an $r\times r$ invertible matrix $F$, an $m \times m$, $\mathbb F$-weighted permutation matrix $P$, and a field automorphism $\sigma: \mathbb F \to \mathbb F$ such that $M'=\sigma(FMP)$, where $P$ also permutes the column labels. Regular matroids are {\em uniquely representable} over any field $\mathbb F$ up to representation equivalence \cite[Corollary 10.1.4]{Oxley}.

Let $\mathcal{M}(E, \mathcal{I})$ be a regular matroid. Up to representation equivalence, $\mathcal{M}$ has a unique representation over $\mathbb{R}$ as a the columns of a totally unimodular (TU) matrix $M$. Let $c_e$ denote the column vector of $M$ corresponding to $e \in E$. 
Replacing a column $c_e$ with $-c_e$ preserves the fact that $M$ is TU and that $M$ represents $\mathcal{M}$. For a graphical matroid represented by its signed incidence matrix less a row, switching the orientation of an edge replaces the corresponding matrix column with its negative.

An {\em oriented matroid} \cite{BLSWZ} abstracts properties of oriented graphs, as well as arrangements of vectors in vector spaces over ordered fields. One way to define oriented matroids is to enhance the circuits of a matroid from merely sets to {\em signed sets}. A signed set $X$ is a set equipped with a partition $X=X_{+}\sqcup X_{-}$, in other words, each element has a sign attached. The {\em support} of a signed set is the underlying set, with signs forgotten. If $X$ is a signed set, let $-X$ denote the signed set identical to $X$ as a set, but with the partition reversed: $(-X)_+=X_-$ and $(-X)_-=X_+$.

\begin{definition}
An {\em oriented matroid} is an ordered pair $(E, \mathcal C)$, where $E$ is a set, called the {\em ground set}, and $\mathcal C$ is a collection of signed sets, called {\em signed circuits}, each supported by a subset of $E$. This data is subject to the axioms:
\begin{enumerate}
\item $\emptyset \notin \mathcal C.$
\item For all $X \in \mathcal C$, $-X \in \mathcal C.$
\item For all $X, Y \in \mathcal C,$ if $X \subseteq Y$ as sets, then ($X=Y$ or $X=-Y$).
\item For all $X,Y \in \mathcal C$, $X\neq -Y$, with an $e\in X_+ \cap Y_-$, there is a $Z\in \mathcal C$ such that $Z_+ \subseteq (X_+ \cup Y_+)\setminus \{ e\}$, and $Z_- \subseteq (X_- \cup Y_-) \setminus \{ e\}$.
\end{enumerate}
\end{definition}

\begin{example}\label{ex:OriMat} The following are two motivating examples for oriented matroids:
\begin{enumerate}
\item If $G$ is an oriented graph, the graphical matroid is naturally equipped with the structure of an oriented matroid. The signed circuits correspond to oriented circuits of the graph, where the edge signs are determined by whether the orientation of each edge in the circuit agrees or disagrees with the orientation of the circuit.

\item If $\mathcal M$ is a regular matroid, choosing a representing TU matrix $M$ gives rise to an oriented matroid whose ground set $\{c_1,...,c_n\}$ is the set of columns of $M$, and whose signed circuits are supported on the circuits of $\mathcal M$. For
each circuit (minimal dependent set of columns) $C=\{c_{e_1},...,c_{e_s}\}$ of $\mathcal M$, and linear dependence $\sum_{k=1}^s \alpha_k c_{e_k}=0$, there is a signed circuit with $C_+=\{c_{e_j}: \, \alpha_j > 0\}$ and $C_-=\{c_{e_j}: \, \alpha_j < 0\}$.
\end{enumerate}
\end{example}

\begin{definition}
Given a regular matroid $\mathcal M$ with a fixed representing TU matrix $M$, an {\em orientation} of $\mathcal M$ is a map $\epsilon: E \to \{-1, 1\}$. We use the notation $e \mapsto \epsilon_e$.
Each orientation of $\mathcal M$ gives rise to an oriented matroid, namely, multiply each column $c_e$ of $M$ by $\epsilon_e$, call this matrix $M^{{\epsilon}}$ and use it to construct an oriented matroid $\calM^{{\epsilon}}$ as above.
\end{definition}

The notions of graph cuts and flows extend naturally to regular matroids \cite{Su-Wagner}:

\begin{definition}
Given a regular matroid $\mathcal M$ with representing TU matrix $M$, $\mathcal{F}(\mathcal{M}):=\ker(M)$ is the {\em vector space of real-valued flows}. A bilinear form on $\mathcal{F}(\mathcal{M})$ is induced by the Euclidean inner product on $\R^E$.
The lattice of integer flows is $\Lambda(\mathcal{M}) := \ker(M) \cap \mathbb{Z}^E$, with the inner product inherited from $\mathcal{F}(\mathcal{M})$. 
\end{definition}

Up to isometry, $\mathcal{F}(\mathcal{M})$ and $\Lambda(\mathcal{M})$ are independent of the choice of representing TU matrix $M$. The genus of $\mathcal M$ and Voronoi cells are defined exactly as before. Next, we define strongly connected orientations:

\begin{definition}\label{def:SCMat}
A regular matroid $\mathcal M$, with orientation given by a choice of representing TU matrix $M$, is {\em strongly connected}\footnote{This notion also appears as {\em totally cyclic} in the matroid literature, for example in \cite{BLSWZ}.} if every $e\in E$ is part of a {\em positive cycle}. In other words, for every $e \in E$, there exists $z_e  \in (\mathbb{Z}_{\geq 0})^E$ such that $e + z_e \in \Lambda(\mathcal{M})$. An orientation $\epsilon$ of $(\mathcal M, M)$ is strongly connected if $(\mathcal M, M^{{\epsilon}})$ is strongly connected. 
\end{definition}

The following proposition follows from the theorems of \cite[Section 3.4]{BLSWZ}; we include a self-contained proof here:

\begin{proposition}
Let $G$ be an oriented graph, $\mathcal M(G)$ the corresponding graphical matroid, and $M$ the signed incidence matrix less a row. Then  $G$ is strongly connected if and only if $(\mathcal M (G), M)$ is strongly connected.
\end{proposition}
\begin{proof}
First assume that $G$ is strongly connected. Then signed incidence matrix $\widetilde{M}$ (with all rows included) determines a linear transformation $\widetilde{M}: \mathbb R^{E} \to \mathbb R^{V}$, with standard bases given by the vertices and edges. Let $c_{e}$ denote the column of $\widetilde{M}$ indexed by $e$.  That is, for $e=(u,v)$, $\widetilde{M}(e)=c_{e}=v-u$. 

By the definition of strongly connected graphs, for every directed edge $e=(u,v)$, there exists a directed path $v \xrightarrow{Z_e} u$.  Let $z_{e}:= \sum_{e'\in Z_{e}} e'$. Observe that $\widetilde{M}(z_e) =  u -  v$, which implies $\widetilde{M}( e +  z_e) = 0$, so $ e +  z_e \in \Lambda(\mathcal{M})$. This remains true after deleting any single row of $\widetilde{M}$, hence, graphical matroids of strongly connected graphs are strongly connected as oriented matroids.

Now assume that $(\mathcal M (G), M)$ is strongly connected as an oriented regular matroid. An element $e \in E$ represents a directed edge $e=(u,v)$ in $G$. The sum $e + z_{e}\in \Lambda(\mathcal M(G))=\Lambda(G)$ is a positive cycle in $\Lambda(G)$. Write $e + z_{e}= \sum_{e' \in E} \alpha_{e'} e'$, where $\alpha_{e'}\in \mathbb Z_{\geq 0}$ for all $e' \in E$, and $\alpha_e > 0$. Being a cycle, $e+z_e$ satisfies {\em Kirkhoff's law}, meaning that for any vertex $w$ of G, the sum of wights of incoming edges to $w$ is the same as the sum of weights of outgoing edges from $w$: $$\sum_{\xrightarrow{e'}w} \alpha_{e'} - \sum_{w \xrightarrow{e''}} \alpha_{e''}=0.$$

To prove that $G$ is strongly connected, we need to show that there is a directed path from $v$ to $u$. Let $T$ be the set of vertices reachable via directed paths from $v$, and $S$ be the set of vertices not reachable. By contradiction, assume that $u \in S$. The sets $T$ and $S$ partition the vertex set of $G$, and by definition, all edges between them are directed from $S$ to $T$. There is at least one such edge, namely $e$.

Consider the sum $$0=\sum_{w\in S} \left( \sum_{\xrightarrow{e'}w} \alpha_{e'} - \sum_{w \xrightarrow{e''}} \alpha_{e''} \right).$$
For each edge $e'$ whose beginning and end are both in $S$, $\alpha_{e'}$ appears in this sum twice, with opposite signs, hence cancels. Furthermore, there are no directed edges from $T$ to $S$. Hence, the sum simplifies to
$$0 = -\sum_{\substack{w \xrightarrow{e''} t \\ w \in S, t \in T}}\alpha_{e''}$$
In this sum the weights of edges from $S$ to $T$ only appear once, with negative signs. All these weights are non-negative, and at least one of them, $\alpha_e$, is positive. Therefore, the sum is negative, a contradiction.
Hence, any edge $e$ is part of a positive circuit and $G$ is strongly connected.
\end{proof}

We are now ready define the poset $\mathcal{SC}$ for regular matroids. 

\begin{definition}
Let $\mathcal M$ be a regular matroid with a fixed representing TU matrix $M$. The elements of $\mathcal{SC}(\mathcal{M})$ are ordered pairs of the form $(\mathcal{N}, D_\mathcal{N})$, with
\begin{itemize}
\item $\mathcal{N}$ an (unoriented) sub-matroid of $\mathcal{M}$, that is, a subset of $E$ with the induced independent sets;
\item $D_\mathcal{N}$ is an orientation of $\mathcal N$, that is, $D_{\mathcal N}$ is obtained from the submatrix of $M$ corresponding to the columns labelled by the ground set of $\mathcal{N}$, and multiplying some columns by $-1$;
\item $(\mathcal N, D_{\mathcal N})$ is strongly connected as an oriented matroid, as in Definition~\ref{def:SCMat}.
\end{itemize}

Define an ordering $(\mathcal{N}, D_\mathcal{N}) \leq (\mathcal{N}', D_{\mathcal{N}'})$ if $\mathcal{N}'$ is a sub-matroid of $\mathcal{N}$ and $D_{\mathcal{N}'}$ is the orientation induced from $D_\mathcal{N}$, that is, the submatrix of $D_\mathcal{N}$ corresponding to the columns in $\mathcal{N'}$, with the same signs. 
\end{definition}

\begin{remark}\label{rmk:Graded}
Note that, just like it is for graphs, $\mathcal{SC}(\mathcal{M})$ is graded by the genus of $\calM$ minus the genus of $\mathcal N$, where $\genus \mathcal{N} =\dim\mathcal{F}(\mathcal{N})$. The proof of this fact is identical to that of Proposition~\ref{prop:PosetGraded} (which is based on \cite{Amini}). The maximal element of $\mathcal{SC}(\mathcal{M})$ is the pair $(\emptyset, \emptyset)$, with genus zero.
\end{remark}

This allows for the generalisation of Amini's theorem to regular matroids:
\begin{theorem}\label{main}
For a finite regular matroid $\mathcal{M}$, $\mathcal{FP}(\mathcal{M}) \cong \mathcal{SC}(\mathcal{M})$ as graded posets.
\end{theorem}

\section{Proof}
\subsection{Preliminaries}
In this section, we discuss the definitions and properties of necessary graph theoretical notions in the more general context of regular matroids. First we establish the following notation: 
\begin{itemize}
\item From here on let $\mathcal{M}(E, \mathcal{I})$ denote a regular matroid, represented by a TU matrix $M$. 
\item Let $\mathcal{F} := \mathcal{F}(\mathcal{M})$ and $\Lambda := \Lambda(\mathcal{M})$. For $\lambda \in \mathcal{F}$ and $e \in E$, let $\lambda_e:= \inner{\lambda, e}$. 
\item For $\lambda \in \mathcal{F}$, let $\supp(\lambda):= \{e \in E\ |\ \lambda_e \neq 0\}$ denote the {\em support} of the flow $\lambda$. Let $\col(\lambda) = \{c_e\ |\ e \in \supp(\lambda)\}$ be the set of columns of $M$ corresponding to $\supp(\lambda)$. Let $M_\lambda$ be the submatrix of $M$ with columns in $\col(\lambda)$. 
\item Let $\genus(\lambda):=\dim(\ker(M_\lambda))$.
\item For $x \in \mathbb{R}$, $\sgn(x) = \begin{cases} +1 & x > 0 \\ \phantom{-}0 & x = 0 \\ -1 & x < 0\end{cases}$.
\end{itemize}

By definition, a flow is always supported on a dependent set; for this reason, graph flows are alternatively referred to as cycles. We need a good notion of circuits for flows. We could say a flow $\lambda \in \calF$ is called a {\em circuit} if it is supported on a circuit, that is, any proper subset of $\col(\lambda)$ is a set of independent vectors. However, with this definition, any real multiple of a circuit flow is also a circuit. The notion of an Eulerian flow helps restrict this to a definition that is more in line with one's graph theoretic image of a circuit.

\begin{definition}
A nonzero flow $\lambda\in\mathcal{F}$ is {\em Eulerian} if $\lambda_e \in \{-1, 0, 1\}$ for all $e \in E$.
\end{definition}
Every flow of a graph supported on a circuit is a scalar multiple of an Eulerian flow. This is also true\footnote{While this is an elementary result, we haven't found a reference and hence include a proof.} for regular matroids:
\begin{proposition}
Let $\lambda \in \calF$ be a flow supported on a circuit. Then $\lambda = k\gamma$ for an Eulerian flow $\gamma$ and $k \in \mathbb{R}$.
\end{proposition}
\begin{proof} The proof is elementary linear algebra.
Let $\col(\lambda) = \{c_1, c_2, \cdots, c_n, c_{n+1}\}$. Since $\lambda $ is supported on a circuit, there are unique {\em nonzero} coefficients $a_1, a_2, \cdots, a_n \in \mathbb{R}$ such that
$$c_{n+1} = \sum_{k=1}^n a_k c_k$$

Since $c_1, c_2, \cdots, c_n$ are linearly independent, there exists a choice of $n$ rows such that the corresponding $n\times n$ submatrix $A$ has a nonzero determinant. As $M$ is a TU matrix, it follows that $\det A = \pm 1$. Let the vectors $d_1, d_2, \cdots, d_{n+1}$ be the vectors obtained from $c_1, c_2, \cdots, c_{n+1}$ by restricting to these rows, so 
$$d_{n+1} = \sum_{k=1}^n a_k d_k$$

For each $k = 1, 2, \cdots, n$, denote by $A^{(k)}$ the matrix obtained by replacing $k$-th column of $A$ with $d_{n+1}$. Note that $\det  A^{(k)} = a_k \det A$.
In addition,  $A^{(k)}$ is a submatrix of $M$ up to permutation of the columns, hence $\det A^{(k)} \in \{-1, 0, 1\}$ as $M$ is TU. As we have established that $a_k \neq 0$, it follows that $a_k = \pm 1$.

Let $\gamma = e_{n+1} - \sum_{k=1}^n a_k e_k$, which is an Eulerian flow as $M\gamma = c_{n+1} - \sum_{k=1}^n a_kc_k = 0$. From the uniqueness of $a_1, a_2, \cdots, a_n$ and the fact that $\lambda \in \Lambda$, it follows that $\lambda$ is a scalar multiple of $\gamma$.
\end{proof}

\begin{definition}
A {\em circuit in } $\Lambda$ is an Eulerian flow  supported on a circuit of $\calM$.
\end{definition}

As a final ingredient, we recall a classical theorem of Tutte on {\em consistent circuit decompositions}, and derive some of its consequences.

\begin{definition}
Two elements of $\lambda, \mu \in \mathbb{R}^E$, in particular two circuits, are {\em consistent} if $\lambda_e\mu_e \geq 0$ for all $e \in E$.
Define the relation $\mu \subseteq \lambda$ if there exists $\gamma \in \mathbb{R}^E$ such that $\lambda = \gamma + \mu$ and $\gamma, \mu$ are consistent with $\lambda$.
\end{definition}

Note that the containment  relation is transitive, and $\mu \subseteq \lambda$ implies $\lambda -\mu \subseteq \lambda$ and $\abs{\lambda_e} \geq \abs{\mu_e}$ for all $e \in E$, in other words, $\supp \mu \subseteq \supp \lambda$.
 
Given a flow $\lambda \in \calF$, if $\lambda_e \geq 0$ for all $e \in E$, then $\lambda$ is called a \textit{positive flow}; circuits in $\Lambda$ that are positive flows are called \textit{positive circuits}. The containment relation is particularly simple for positive flows: given two positive flows $\lambda$ and $\mu$, $\mu \subseteq \lambda$ if and only if $\mu_e \leq \lambda_e$ for all $e \in E$.

\begin{theorem}\label{consistent}
\cite[Theorem 5.43]{Tutte} Let $\lambda \in \Lambda$. There exists circuits $\gamma_1, \gamma_2, \cdots, \gamma_n$ such that $\lambda = \displaystyle\sum_{i=1}^n \gamma_i$, and $\gamma_i \subseteq\lambda$ for all $i = 1, 2, \cdots, n$. Such a decomposition is called a consistent circuit decomposition for $\lambda$.
\end{theorem}

Note that if $\lambda$ is Eulerian, then the circuits in a consistent decomposition for $\lambda$ are disjoint.

\begin{corollary}\label{arbitrary}
For $\lambda \in \Lambda$ and any circuit $\gamma \subseteq \lambda$, there exists a consistent circuit decomposition which includes $\gamma$.
\end{corollary}

\begin{proof}
If $\lambda$ is the zero flow, the statement is vacuously true. Otherwise, given a circuit $\gamma \subseteq \lambda$,  by definition $\lambda - \gamma \subseteq \lambda$, and by Theorem~\ref{consistent},  $\lambda-\gamma=\displaystyle\sum_{i=1}^n\gamma_i$.
Thus, $\lambda = \gamma+\displaystyle\sum_{i=1}^n\gamma_i$ is a consistent circuit decomposition of $\lambda$ which includes $\gamma$.
\end{proof}

\begin{corollary}\label{circuit exists}
For any $e \in \supp(\lambda)$, there exists a circuit $\gamma \subseteq \lambda$ such that $e \in \supp(\gamma)$.
\end{corollary}

\begin{proof}
Let $\lambda = \displaystyle\sum_{i=1}^n \gamma_i$ be a consistent circuit decomposition. Then
$\supp(\lambda) = \bigcup_{i=1}^n \supp(\gamma_i).$
Thus, for every $e \in \supp(\lambda)$, there exists a circuit $\gamma_i \subseteq \lambda$ such that $e \in \supp(\gamma_i)$.
\end{proof}

\begin{lemma}\label{union}
Let $\lambda, \mu \in \mathcal{F}$. Then there exists $k \in \mathbb{R}^{>0}$ such that $\supp(k\lambda + \mu) = \supp(\lambda)\cup \supp(\mu)$.
\end{lemma}

\begin{proof} It is clear that $\supp(k\lambda +\mu)\subseteq \supp(\lambda) \cup\supp(\mu)$ for any $k$.
For the other direction, one needs to choose sufficiently large $k$: let $$k = \max\left\{\dfrac{\abs{\mu_e}}{\abs{\lambda_e}}\ \middle|\ e \in \supp(\lambda)\right\}+1.$$

Let $e \in \supp(\lambda)\cup \supp(\mu)$.
If $\lambda_e = 0$, then $e \in \supp(\mu)$, hence $(k\lambda+\mu)_e = \mu_e \neq 0$ which implies $e \in \supp(k\lambda + \mu)$. Otherwise, if $\lambda_e \neq 0$, then
\begin{equation*}
\abs{(k\lambda+\mu)_e} = \abs{k\lambda_e + \mu_e}
\geq k\abs{\lambda_e} - \abs{\mu_e}
\geq \abs{\lambda_e} > 0,
\end{equation*}
which also implies $e \in \supp(k\lambda + \mu)$. Thus, $\supp(\lambda)\cup \supp(\mu) \subseteq \supp(k\lambda + \mu)$.
\end{proof}

The following is a strengthening of Theorem \ref{consistent}:

\begin{theorem}\label{basis}
Let $\lambda\in \mathcal{F}$ be nonzero. Then there exists a basis of $\ker(M_\lambda)$ consisting of circuits consistent with $\lambda$.
\end{theorem}
\begin{proof}
Without loss of generality, reorient $M_\lambda$ by replacing each $c_e \in \col(\lambda)$ with $\sgn(\lambda_e)c_e$. With respect to this orientation, $\lambda$ is a positive flow and the statement is equivalent to finding a positive circuit basis of $\ker(M_\lambda)$. 

We proceed by induction on $\genus(\lambda)$. If $\genus(\lambda) = 1$, then $\ker(M_\lambda)$ is spanned by a single circuit $\gamma$. Since $\lambda \in \ker(M_\lambda)$, $\lambda = a\gamma$ for some $a \in \mathbb{R}$. As $\lambda$ is a positive flow, $\sgn(a)\gamma$ is a positive circuit.

Assume that $\genus(\lambda) \geq 2$. Let $\mu \in \ker(M_\lambda)$ be such that $|\supp(\mu)|$ is maximal but not equal to $|\supp(\lambda)|$. 

\medskip
We prove the following statement first:
\begin{equation}\label{star}
\text{If $\xi \in \ker(M_\lambda)$ such that $\supp(\xi) \not\subseteq \supp(\mu)$, then }
\supp(\lambda)\setminus \supp(\mu)\subseteq \supp(\xi). \tag{$\star$}
\end{equation}
By Lemma \ref{union} there exists $k \in \mathbb{R}^{>0}$ such that $\supp(k\xi + \mu) = \supp(\xi) \cup \supp(\mu)$.  Thus,
$$|\supp(k\xi+\mu)| = |\supp(\xi) \cup \supp(\mu)| > |\supp(\mu)|.$$
By the maximality of $\supp(\mu)$, this implies that $\supp(\xi) \cup \supp(\mu) = \supp(\lambda)$. This implies \eqref{star}.

\medskip
We now prove the inductive step. Fix any $e_0 \in \supp(\lambda)\setminus\supp(\mu)$. By Corollary \ref{circuit exists}, there exist a positive circuit $\gamma$ such that $e_0 \in \supp(\gamma)$.  Hence, by \eqref{star}, we have
$\supp(\lambda)\setminus\supp(\mu) \subseteq \supp(\gamma).$

For an arbitrary flow $\delta \in \ker(M_\lambda)$, define $\delta' = \delta - \delta_{e_0}\gamma$. Notice that $\delta'_{e_0} = 0$, which implies
$$\supp(\lambda)\setminus \supp(\mu)\not\subseteq \supp(\delta').$$
This, by the contrapositive of \eqref{star}, implies that $\supp(\delta') \subseteq \supp(\mu)$.

As $\mu \subseteq \lambda$, we have $\ker(M_\mu) \subseteq \ker(M_\lambda)$. Since $\supp(\delta') \subseteq \supp(\mu)$, we have $\delta' \in \ker(M_\mu)$. Furthermore, $\gamma \not \in \ker(M_\mu)$ as $\supp(\gamma) \not \subseteq \supp(\mu)$. Thus for any $\delta \in \ker(M_\lambda)$, there exists a unique decomposition $\delta = \delta' + \delta_{e_0}\gamma$, hence
$$\ker(M_\lambda) = \ker(M_\mu) \oplus \mathbb{R}\gamma.$$

By the induction hypothesis, there is a positive circuit basis of $\ker(M_\mu)$. The positive circuit basis of $\ker(M_\lambda)$ is the union of this basis with $\gamma$.
\end{proof}

\subsection{Supporting lemmas}
With the above preliminaries in place, the supporting Lemmas of \cite{Amini} translate readily to the language of matroids. We include their proofs for completeness.
 
\begin{lemma}\label{lem:Halfway} \cite[Lemma 3]{Amini}
Given any $\lambda \in \Lambda$, $V_\lambda \cap V_0 \neq \emptyset$ if and only if $\lambda/2  \in V_\lambda \cap V_0$.
\end{lemma}
\begin{proof}
Indeed if $\lambda/2 \in V_\lambda \cap V_0$, then $V_\lambda \cap V_0 \neq \emptyset$. For the other direction, recall that
$$V_\lambda := \{x \in \calF\ |\ \norm{x - \lambda} \leq \norm{x - \mu},\ \forall \mu \in \Lambda\}$$
Since $V_ \lambda \cap V_0$ is nonempty, we can choose $x \in V_\lambda\cap V_0$. Then for any $\mu \in \Lambda$, we have
$$\norm{x} = \norm{x-\lambda}, \qquad \norm{x} \leq \norm{x - \mu}, \qquad \norm{x} \leq \norm{x - (\lambda - \mu)}.$$
After squaring both sides and expanding the inner products, we obtain
$$\norm{\lambda}^2 = 2\inner{x, \lambda}, \qquad 2\inner{ x, \mu}\leq  \norm{\mu}^2, \qquad 2\inner{ x, \lambda-\mu } \leq \norm{\lambda - \mu}^2.$$
Therefore,
\begin{align*}
\norm{\mu}^2 + \norm{\lambda-\mu}^2 &\geq 2\inner{x, \mu} + 2\inner{x, \lambda-\mu} = 2\inner{ x, \lambda } = \norm{\lambda}^2.
\end{align*}
This is equivalent to
$$\norm{\lambda/2-\mu} \geq \norm{\lambda/2 - 0} = \norm{\lambda/2 - \lambda}.$$
As $\mu$ was arbitrary, it follows that $\lambda/2 \in V_0\cap V_\lambda$.
\end{proof}
\begin{lemma}\label{Eulerian} \cite[Lemma 3]{Amini}
Given any $\lambda \in \Lambda$, $V_\lambda \cap V_0 \neq \emptyset$ if and only if $\lambda$ is Eulerian.
\end{lemma}

\begin{proof}
First assume $V_\lambda \cap V_0 \neq \emptyset$; this is equivalent to $\lambda/2 \in V_\lambda \cap V_0$ by Lemma~\ref{lem:Halfway}.
Let $\gamma$ be a circuit such that $\gamma \subseteq \lambda$. For each $e \in \supp(\gamma)$, we have $\abs{\gamma_e} = 1$ and $\lambda_e\gamma_e \geq 0$, and so
\begin{align*}
\inner{\lambda, \gamma} = \displaystyle\sum_{e \in \supp(\gamma)} \lambda_e\gamma_e =  \displaystyle\sum_{e \in \supp(\gamma)} |\lambda_e| \quad \geq \quad  \norm{\gamma}^2,
\end{align*}
where equality holds if and only if $|\lambda_e| = 1$ for all $e \in \supp(\gamma)$.

As $\lambda/2 \in V_\lambda \cap V_0$,  we have $\norm{\lambda/2 - \gamma}^2 \geq \norm{\lambda/2}^2$. This is equivalent to $\norm{\gamma}^2 \geq \inner{\gamma, \lambda}$. 
Combined with the previous inequality, we have $\inner{ \lambda, \gamma } = \norm{\gamma}^2$, which implies $\abs{\lambda_e} = 1$ for all $e \in \supp(\gamma)$. 

Furthermore, as every $e \in \supp(\lambda)$ is contained in some circuit $\gamma$ with $\gamma \subseteq \lambda$, we can conclude that $\abs{\lambda_e} = 1$ for all $e \in \supp(\lambda)$. Thus, $\lambda$ must be Eulerian.

Now, assume $\lambda$ is Eulerian. Then for any $\mu \in \Lambda$, 
\begin{align*}
\norm{\lambda/2 - \mu}^2 &= \sum_{e \in \supp(\lambda)} (\mu_e - \lambda_e/2)^2 + \sum_{e \in \supp(\mu)\setminus\supp(\lambda)} \mu_e^2\\
&\geq \sum_{e \in \supp(\lambda)} (1/2)^2\\
&= \norm{\lambda/2}^2,
\end{align*}
where the inequality uses that $\lambda$ is Eulerian.
Thus $\lambda/2 \in V_\lambda \cap V_0$, and so $V_\lambda \cap V_0 \neq \emptyset$.
\end{proof}

\begin{lemma}\label{codimension} \cite[Lemma 4]{Amini}
For any Eulerian flow $\lambda$, $V_\lambda \cap V_0$ lies in an affine plane with codimension equal to $\genus(\lambda)$. In particular, if the codimension is one, then $\lambda$ is a circuit.
\end{lemma}

\begin{proof}
Let $\gamma$ be any circuit such that $\gamma \subseteq \lambda$. If the $V_\lambda \cap V_0 \neq \emptyset$, then by Lemma \ref{Eulerian}, $\lambda$ is Eulerian and by Theorem~\ref{consistent} there exists a disjoint consistent circuit decomposition $\lambda = \gamma_1 + \cdots + \gamma_n$ with $\gamma = \gamma_1$, and $\inner{ \gamma_i, \gamma_j} = 0$ for $i \neq j$. Thus
$$\norm{\lambda}^2 = \norm{\displaystyle\sum_{i=0}^n \gamma_i}^2 = \displaystyle\sum_{i=0}^n \norm{\gamma_i}^2$$

Let $x \in V_\lambda \cap V_0$. Then $\norm{x-\lambda} = \norm{x}$ and  $\norm{x} \leq \norm{x-\gamma_i}$, for all $i = 1, 2, \cdots, n$. These are equivalent to $\norm{\lambda} = 2\inner{x, \lambda}$ and $2\inner{ x, \gamma_i } \leq \norm{\gamma_i}^2$. Putting these together,
\begin{align*}
\norm{\lambda}^2 = 2\inner{ x, \lambda } = 2\sum_{i = 1}^n \inner{ x, \gamma_i } \; \leq \; \sum_{i=1}^n\norm{\gamma_i}^2 = \norm{\lambda}^2
\end{align*}

It follows that each inequality $2\inner{ x, \gamma_i } \leq \norm{\gamma_i}^2$ is an equality. In particular for $i=1$, we have $2\inner{x, \gamma} = \norm{\gamma}^2$. As $\gamma \subseteq \lambda$ was arbitrary, we have proven that $2\inner{x, \gamma} = \norm{\gamma}^2$ for all $\gamma \subseteq \lambda$.

Let $\genus(\lambda) = g$. Let $\alpha_1, \alpha_2, \cdots, \alpha_g$ be the circuit basis of $M_\lambda$ such that $\alpha_i \subseteq \lambda$ for all $i = 1, 2, \cdots, g$, which exists by Theorem \ref{basis}. Then $V_\lambda \cap V_0$ is contained in the plane
$$\{x \in \mathcal{F}\ |\ 2\inner{x, \alpha_i} = 2\norm{\alpha_i}^2,\ \text{for all }i=1, 2, \cdots, g\}$$
This plane has codimension $g$ as it is defined by $g$ independent linear equations. Note that the linear equations obtained are also independent of the rows of $M$.
\end{proof}

In view of Lemma~\ref{codimension} we fix the following notation.
Let $\Xi$ denote the set of circuits in $\Lambda$, that is, $\Xi := \{\gamma \in \Lambda\ |\ \gamma \text{ is a circuit}\}$.
For any circuit $\gamma \in \Xi$, let $F_\gamma := \{ x\in \mathcal{F}\ |\ 2\inner{x, \gamma} = \norm{\gamma}^2\}$ be the hyperplane which contains $V_\gamma \cap V_0$.

\begin{lemma}\label{Face consistent} \cite[Lemma 6]{Amini} If $\lambda,\mu \in \Xi$ are circuits and $V_0\cap F_{\lambda}\cap F_{\mu}\neq \emptyset$,
then $\lambda$ and $\mu$ are consistent.
\end{lemma}
\begin{proof}
Define
$S:\Lambda \to \mathbb{Z}^+$ by setting $$S(\lambda)= \displaystyle\sum_{e\in\supp(\lambda)} \abs{\lambda}.$$
Note that $S(\lambda) = \norm{\lambda}^2$ for an Eulerian flow $\lambda$. By the triangle inequality, $S(\lambda)+S(\mu) \geq S(\lambda+\mu)$, where equality holds if and only if $\lambda, \mu$ are consistent. Thus, it is enough to prove that $S(\lambda)+S(\mu)=S(\lambda+\mu).$

Let $x\in V_0\cap F_{\lambda}\cap F_{\mu}$. Then $2\inner{x, \lambda} = \norm{\lambda}^2 = S(\lambda)$ and $2\inner{x, \mu} = \norm{\mu}^2 = S(\mu)$. Thus,
$$2\inner{x, \lambda+\mu} = S(\lambda) + S(\mu).$$
Consider a consistent circuit decomposition $\lambda + \mu = \displaystyle\sum_{i=1}^n \gamma_i$, which exists by Theorem \ref{consistent}.
As $x \in V_0$, we must have $\norm{x - \gamma_i} \geq \norm{x}$ for all $i = 1, 2, \cdots, n$. This is equivalent to
$$S(\gamma_i) = \norm{\gamma_i}^2 \geq 2\inner{x, \gamma_i}$$
As each circuit $\gamma_i$ are consistent with $\lambda + \mu$, we also have
$$\displaystyle\sum_{i=1}^n S(\gamma_i) = S\left(\displaystyle\sum_{i=1}^n \gamma_i\right) = S(\lambda + \mu)$$
In conclusion,
\begin{align*}
S(\lambda) + S(\mu)= 2\inner{x, \lambda + \mu} = \sum_{i=1}^n 2\inner{x, \gamma_i} \leq \sum_{i=1}^n S(\gamma_i) = S(\lambda + \mu)\leq S(\lambda) + S(\mu)
\end{align*}
which implies $S(\lambda + \mu) = S(\lambda) + S(\mu)$, and the statement follows.
\end{proof}

\subsection{Proof of Theorem \ref{main}}
The proof strategy is to define a map $\phi:\mathcal{FP}(\mathcal{M}) \to \mathcal{SC}(\mathcal{M})$, we briefly preview this construction. For a face $F$, the map $\phi$ assigns the union of the supports of circuits corresponding to the codimension-one faces which contain $F$. There is a natural orientation for this sub-matroid induced by the same circuits, which is strongly connected. We prove that this map is order-preserving, injective, surjective, and takes codimension to genus: in other words, it is an isomorphism of graded posets.
\medskip

We begin with a piece of notation. For $\lambda \in \Lambda$, let $\mathcal{N}_\lambda$ be a sub-matroid of $\mathcal{M}$ with ground set $\supp(\lambda)$. Let $D_\lambda = \epsilon_\lambda$ be the orientation on $\mathcal{N}_\lambda$ defined by $\epsilon_\lambda(e) = \sgn(\lambda_e)$.

\begin{lemma}
$D_\lambda$ is a strongly connected orientation on $\mathcal{N}_\lambda$. In other words, $(\mathcal{N}_\lambda, D_\lambda) \in \mathcal{SC}(\mathcal M)$.
\end{lemma}
\begin{proof}
With respect to the orientation $D_\lambda$, $\lambda$ is a positive cycle. Hence, any $e \in \supp(\lambda)$ is part of a positive cycle, namely $\lambda$. Thus, $D_\lambda$ is a strongly connected orientation on $\mathcal{N}_\lambda$.
\end{proof}

We now define a map $\phi:\mathcal{FP}(\mathcal{M}) \to \mathcal{SC}(\mathcal{M})$; we will prove that $\phi$ is an isomorphism. For $F \in \mathcal{FP}(\mathcal{M})$, let $\mathcal{U}(F)$ denote the set of circuits corresponding to the codimension one faces which contain $F$, that is, $$\mathcal{U}(F) = \{\gamma \in \Xi\ |\ F \subseteq F_\gamma\}.$$
Define
$$\phi(F) := \left(\bigcup\limits_{\lambda \in \mathcal{U}(F)}\mathcal{N}_\lambda, \bigcup\limits_{\lambda \in \mathcal{U}(F)} D_\lambda \right). $$
Here,  $\displaystyle\bigcup\limits_{\lambda \in \mathcal{U}(F)}\mathcal{N}_\lambda$ denotes the sub-matroid of $\calM$ with ground set $\displaystyle\bigcup\limits_{\lambda \in \mathcal{U}(F)}\supp(\lambda)$.

By Lemma \ref{Face consistent}, for any $\lambda, \mu \in \mathcal{U}(F)$, $D_\lambda$ and $D_\mu$ are consistent. As $D_\lambda$ is a strongly connected orientation of $\mathcal{N}_\lambda$, it follows that $\bigcup\limits_{\lambda \in \mathcal{U}(F)} D_\lambda$ is a strongly connected orientation of $\bigcup\limits_{\lambda \in \mathcal{U}(F)}  \mathcal{N}_\lambda$. Thus, $\phi(F)\in \mathcal{SC}(\mathcal{M}),$ as claimed. 
It remains to show that $\phi$ is an isomorphism of graded posets.

\begin{lemma}\label{faceiff} \cite[Lemma 7]{Amini}
Given $\lambda \in \Xi$ a circuit, $\lambda \in \mathcal{U}(F)$ if and only if $(\mathcal{N}_\lambda, D_\lambda) \geq \phi(F)$ in the partial order of $\mathcal{SC}(\mathcal{M})$.
\end{lemma}

\begin{proof}
The only if direction follows from the definition of $\phi(F)$ and the partial ordering of $\mathcal{SC}(\mathcal{M})$.
For the if implication, let $\lambda$ be a circuit such that $(\mathcal{N}_\lambda, D_\lambda)\geq \phi(F)$; we aim to prove that $\lambda \in \mathcal{U}(F).$

Define the flow $\mu := \displaystyle\sum_{\gamma \in \mathcal{U}(F)}\gamma$. Note that this is expression a consistent circuit decomposition as all flows $\gamma$ in the sum are consistent due to Lemma \ref{Face consistent}. So we can write
$$D_\mu = \displaystyle\bigcup\limits_{\gamma \in \mathcal{U}(F)} D_\gamma, \qquad \supp(\mu) = \displaystyle\bigcup\limits_{\gamma \in \mathcal{U}(F)}\supp(\gamma)$$
In other words, $\phi(F)$ is the strongly connected orientation induced by the flow $\mu$.

Choose another consistent circuit decomposition of $\mu = \sum_{i=1}^n \gamma_i$, this time such that $\gamma_1 = \lambda$, which exists due to Corollary \ref{arbitrary}.

Let $x \in F$. As $F \subseteq F_\gamma$ for all $\gamma \in \mathcal{U}(F)$, we have $2\inner{x, \gamma} = \norm{\gamma}^2 = S(\gamma)$. We have
\begin{align*}
2\inner{x, \mu} \; =\; \sum_{\gamma\in\mathcal{U}(F)}2\inner{x, \gamma} \; =\; \sum_{\gamma\in\mathcal{U}(F)} S(\gamma) \; =\; S(\mu)
\end{align*}
However, as $x\in F \subseteq V_0$, we have $\norm{x}^2 \leq \norm{x - \gamma_i}^2$ for all $i = 1, 2, \cdots, n$. This is equivalent to
$$S(\gamma_i) = \norm{\gamma_i}^2 \geq 2\inner{x, \gamma_i}$$
Putting these together,
\begin{align*}
S(\mu) = 2\inner{x, \mu} &= 2\sum_{i=1}^n \inner{x, \gamma_i} \leq \sum_{i=1}^n S(\gamma_i) = S(\mu)
\end{align*}
Thus, the inequality in the middle is an equality:
$$2\sum_{i=1}^n \inner{x, \gamma} = \sum_{i=1}^n S(\gamma_i)$$

Therefore, each inequality $2\inner{ x, \gamma_i} \leq \norm{\gamma_i}^2$ is an equality. In particular, for $i=1$, 
$2\inner{x, \lambda} = \norm{\lambda}^2.$
Hence, $x \in F_\lambda$, so $F \subseteq F_\lambda$, so $\lambda \in \mathcal{U}(F)$.
\end{proof}

\begin{corollary}\label{cor:inj} \cite[Lemma 8]{Amini}
$\phi$ is injective and order-preserving.
\end{corollary}

\begin{proof}
Let $F_1, F_2 \in \mathcal{FP}$ such that $\phi(F_1) = \phi(F_2)$. Let $\lambda \in \mathcal{U}(F_1)$. By Lemma \ref{faceiff}, this is if and only if $(\mathcal{N}_\lambda, D_\lambda) \geq \phi(F_1) = \phi(F_2)$. Again by Lemma \ref{faceiff}, this is if and only if $\lambda \in \mathcal{U}(F_2)$. Thus, $\mathcal{U}(F_1) = \mathcal{U}(F_2)$. Any face $F$ of $V_0$ is obtained as the intersection of the one-codimension faces which contain F; therefore, we conclude that $F_1 = F_2$, and $\phi$ is injective.

Now let $F_1 \subseteq F_2$ be two faces in $\mathcal{FP}$. By definition, $\mathcal{U}(F_2) \subseteq \mathcal{U}(F_1)$. This implies
$$\bigcup\limits_{\lambda\in\mathcal{U}(F_2)}\supp(\lambda) \subseteq \bigcup\limits_{\lambda\in\mathcal{U}(F_1)}\supp(\lambda), \qquad \qquad \bigcup\limits_{\lambda\in\mathcal{U}(F_2)}D_\lambda \subseteq \bigcup\limits_{\lambda\in\mathcal{U}(F_1)}D_\lambda$$
Thus, by the definition of ordering on $\mathcal{SC}$, $\phi(F_1) \leq \phi(F_2)$.
\end{proof}
\begin{lemma}\label{flowmatroid}
Any element of $\mathcal{SC}$ can be represented as $(\mathcal{N}_\mu, D_\mu)$ for some $\mu \in \Lambda$.
\end{lemma}
\begin{proof}
For an element $(\mathcal{N}, D) \in \mathcal{SC}$, consider the sum of all circuits which concur with the orientation $D$ and use edges in $E(\mathcal{N})$. That is,
$$\mu = \displaystyle\sum_{\substack{\gamma \in \Xi\\ (\mathcal{N}_\gamma, D_\gamma) \geq (\mathcal{N}, D)}}\gamma.$$

As $D$ is a strongly connected orientation $\mathcal{N}$, for every $e \in E(\mathcal{N})$, there exists a circuit $\gamma$ such that $(\mathcal{N}_\gamma, D_\gamma) \geq (\mathcal{N}, D)$ and $e \in \supp(\gamma)$. As each of these circuits are mutually consistent in their orientations, their sum is a flow with the orientation $D$ and ground set $E(\mathcal{N})$. Thus $(\mathcal{N}, D) = (\mathcal{N}_\mu, D_\mu)$ for the flow $\mu$.
\end{proof}

\begin{definition}\label{face}
For any $\mu \in \Lambda$ (not necessarily a circuit), let
$$F_\mu = \bigcap\limits_{\substack{\gamma \in \Xi\\(\mathcal{N}_\gamma, D_\gamma) \geq (\mathcal{N}_\mu, D_\mu)}}F_\gamma.$$
\end{definition}

We will show that $\phi$ is surjective by showing that $\phi(F_\mu) = (\mathcal{N}_\mu, D_\mu)$.

\begin{lemma}\label{vertex} \cite[Lemma 9]{Amini}
Let $(\mathcal{N}_\mu, D_\mu) \in \mathcal{SC}$ be such that $\genus(\mu)=\genus(\mathcal{M})$. Then $F_\mu$ consists of a single vertex of $V_0$, and all vertices of $V_0$ are of this form.
\end{lemma}

\begin{proof}
As $F_\mu$ is the intersection of $g:=\genus(\mathcal{M})$ independent hyperplanes, it is contained in an affine plane of codimension $g$. Thus, $F_\mu$ has dimension zero and consists of a single point.

Since $\genus(\mu) = \genus(\mathcal{M})$, $\ker(M_\mu) = \mathcal{F}$. By Theorem \ref{basis}, there exists a basis $\Gamma$  for $\mathcal{F}$, consisting of circuits $\gamma \subseteq \mu$. By Definition \ref{face}, $F_\mu \subseteq F_\gamma$ for all $\gamma \in \Gamma$. Thus, for all $\gamma \in \Gamma$, we have that $2\inner{F_\mu, \gamma} = \norm{\gamma}^2$.

For any $\lambda \in \Lambda$, consider the decomposition
$$\lambda = \sum_{\gamma\in \Gamma} a_\gamma\gamma,\quad \text{ where } a_\gamma \in \mathbb{R}, \gamma \in \Gamma.$$

Thus,
\begin{align*}
2\inner{F_\mu, \lambda} = \sum_{\gamma \in \Gamma} 2a_\gamma \inner{F_\mu, \gamma} &= \sum_{\gamma \in \Gamma} a_\gamma \norm{\gamma}^2 \quad \leq \quad \sum_{e \in \supp(\lambda)} \abs{\lambda_e} = S(\lambda) \leq \norm{\lambda}^2
\end{align*}

Comparing the left and the right, we obtain
$2 \inner{F_\mu, \lambda}\leq \norm{\lambda}^2$, or equivalently, $\norm{F_\mu} \leq \norm{F_\mu-\lambda}$. As $\lambda \in \Lambda$ was arbitrary, this means $F_\mu \in V_0$, hence $F_\mu$ is a vertex of $V_0$.

It remains to show that all vertices of $V_0$ are of this form. Let $v$ be an arbitrary vertex of  $V_0$. Consider $\phi(\{v\}) = (\mathcal{N}, D)$. By Lemma \ref{flowmatroid}, there exists $\mu \in \Lambda$ such that $(\mathcal{N}, D) = (\mathcal{N}_\mu, D_\mu)$. Then for every $\gamma \in \Xi$ such that $(\mathcal{N}_\gamma, D_\gamma) \geq (\mathcal{N}_\mu, D_\mu) = \phi(\{v\})$,  Lemma~\ref{faceiff} implies that $\gamma \in \mathcal{U}(F)$, that is, $\{v\} \subseteq F_\gamma$. Thus,
$$\{v\} \subseteq \bigcap\limits_{\substack{\gamma \in \Xi\\(\mathcal{N}_\gamma, D_\gamma) \geq (\mathcal{N}_\mu, D_\mu)}}F_\gamma \quad = \quad F_\mu.$$
Furthermore, since $v$ has codimension $g$, there exists $g$ faces of $V_0$ of codimension one which intersect at $v$. By Lemma \ref{Face consistent}, these faces are given by $g$ consistent circuits $\alpha_1, \alpha_2, \cdots, \alpha_g$. This is equivalent to
$$\{v\} = \bigcap_{k=1}^g F_{\alpha_k}$$
Thus we have $\alpha_k \in \mathcal{U}(F)$ for $k = 1, 2, \cdots, g$. By Lemma \ref{faceiff}, we have $(\mathcal{N}_{\alpha_k}, D_{\alpha_k}) \geq (\mathcal{N}_\mu, D_\mu) = \phi(\{v\})$, and
$$F_\mu \subseteq \bigcap_{k=1}^g F_{\alpha_k} = \{v\}.$$
In conclusion, $\{v\} = F_\mu$ as claimed.
\end{proof}

In light of Lemma~\ref{vertex}, for $\lambda\in \Lambda$ such that $\genus(\lambda)=\genus(\mathcal{M})$, we denote $F_\lambda$ -- which consists of a single vertex -- by $v_\lambda$.

\begin{lemma}\label{lem:surj} \cite[Lemma 10]{Amini}
Let $(\mathcal{N}_\mu, D_\mu) \in \mathcal{SC}$. The convex hull of the set of vertices
$$S_\mu=\{v_\lambda\ |\ \lambda \in \Lambda,\ \genus(\lambda) = \genus(\mathcal{M}),\ (\mathcal{N}_\mu, D_\mu) \geq (\mathcal{N}_\lambda, D_\lambda)\}$$
defines a face $F$ of $V_0$ with the property $\phi(F) = (\mathcal{N}_\mu, D_\mu)$. Hence, $\phi$ is surjective.
\end{lemma}

\begin{proof}
For any $\lambda \in \Lambda$ with $(\mathcal{N}_\mu, D_\mu) \geq (\mathcal{N}_\lambda, D_\lambda)$, we have  $F_\lambda \subseteq F_\mu$ by Definition \ref{face}. In particular, when $\genus(\lambda) = \genus(\mathcal{M})$, $F_\lambda = \{v_\lambda\} \subseteq F_\mu$ by Lemma~\ref{vertex}. Thus, every vertex of $F$ is contained in $F_\mu$. By the convexity of $V_0$, we have that $F \subseteq F_\mu$.

In fact, $F$ is the face of $V_0$ given by $F=F_\mu \cap V_0$. To show this, recall that $F_\mu \cap V_0$ is a face of $V_0$ for all $\mu$ because it is by definition an intersection of codimension-one faces. In addition, all vertices $v_\lambda$ of $F_\mu \cap V_0$ have the property that $\genus(\lambda) = \genus(\mathcal{M})$, and  $(\mathcal{N}_\mu, D_\mu) \geq (\mathcal{N}_\lambda, D_\lambda)$: these are exactly the vertices $v_\lambda$ included in $S_\mu$. Therefore, $F= F_\mu \cap V_0$ and $F$ is a face of $V_0$. 

It remains to show $\phi(F) = (\mathcal{N}_\mu, D_\mu)$. By definition of $\mathcal{N}_\gamma$, 
$$\mathcal{U}(F) = \{\gamma \in \Xi\ |\ F \subseteq F_\gamma\} = \{\gamma \in \Xi\ |\  (\mathcal{N}_\gamma, D_\gamma) \geq (\mathcal{N}_\mu, D_\mu)\},$$
therefore, by definition of the partial order on $\mathcal{SC}$,
$$\phi(F) = \left(\bigcup\limits_{\gamma \in \mathcal{U}(F)} \mathcal{N}_\gamma, \bigcup\limits_{\gamma\in\mathcal{U}(F)}D_\gamma\right) = (\mathcal{N}_\mu, D_\mu).$$
\end{proof}

\begin{corollary}\label{cor:genus}
For any $F \in \mathcal{FP}$, $\codim(F) = \genus(\phi(F))$.
\end{corollary}

{\em Proof for Theorem \ref{main}.} We defined a map $\phi:\mathcal{FP}(\mathcal{M}) \to \mathcal{SC}(\mathcal{M})$. By Corollary~\ref{cor:inj}, $\phi$ is injective and order preserving. By Lemma~\ref{lem:surj}, $\phi$ is surjective. By Corollary~\ref{cor:genus}, $\phi$ is an isomorphism of graded posets. \qed

\section{Dual Result}
Amini \cite{Amini} presents a dual result and proof for integer cuts of graphs.
In this section we extend this result to regular matroids, and show that it follows Theorem~\ref{main} via matroid duality. 
This theorem, in a more general context, later appeared in \cite{AE20} and became an important tool in describing all stable limits of a family of line bundles along a degenerating family of curves \cite{AE20-2, AE21}.

\subsection{Amini's Theorem for integer cuts} We begin by reviewing the result  \cite{Amini}  for integer cuts of graphs. Given a graph $G$ with a chosen orientation, define the {\em cut space}  $\calC(G)=\Row(M_G)$, where $M_G$ is the signed incidence matrix of $G$. The {\em lattice of integer cuts} of $G$ is defined as $\Gamma(G)=\calC(G) \cap \Z^{E(G)}$, a full-rank sub-lattice in $\calC(G)$. Both $\calC(G)$ and $\Gamma(G)$ inherit inner products from the Euclidean inner product of $\Z^{E(G)}$. Note that the isometry type of $\Gamma(G)$ does not depend on the chosen orientation of $G$. 
Let $\mathcal{FP}^c$ denote the face poset of the Voronoi cell (at zero) of  $\Gamma(G)$. 

A subgraph $H$ of $G$ is called a {\em cut subgraph} \cite{Amini} if there exists an {\em ordered} partition $V(G)=V_1\sqcup ... \sqcup V_s$ of the vertex set of $G$, such that the edges of $H$ are precisely the edges of $G$ which connect separate sets of the partition. That is, and edges within each partition set are omitted. An orientation of a cut subgraph $H$ is called {\em coherent acyclic}, if each for each edge $e=(u,v)$ in $H$, we have $u \in V_i$ and $v \in V_j$ such that $i < j$. Let $(H_1, D_1)$ and $(H_2, D_2)$ be pairs where $H_i$ are cut subgraphs, and $D_i$ is a coherent acyclic orientation for $H_i$ for $i=1,2$. In \cite{Amini} a partial order on these pairs is defined by setting  $(H_2, D_2)\leq (H_1, D_1)$ if $H_1$ is a subgraph of $H_2$ and $D_1$ is induced from $D_2$. The resulting poset is denoted $\mathcal{CAC}$.

\begin{theorem}\label{thm:cuts}\cite[Theorem 14]{Amini}
The posets $\mathcal{FP}^c$ and $\mathcal{CAC}$ are isomorphic.
\end{theorem}

To generalise Theorem ~\ref{thm:cuts} to regular matroids, we first recall the definition of the lattice of integer cuts for a regular matroid -- see for example \cite[Section 2.3]{Su-Wagner} for more detail -- and define coherent acyclic orientations in this context.

\begin{definition}
Given a regular matroid $\mathcal{M}$ represented by a TU matrix $M$, the lattice of integer cuts of $\mathcal M$, denoted $\Gamma(\mathcal M)$, is defined as
$\Gamma(\mathcal{M}) =\Row(M) \cap \mathbb{Z}^E$, with the standard inner product inherited from $\R^E$.
\end{definition}

Note that the isometry type of $\Gamma(\mathcal M)$ does not depend on the choice of $M$, and in the case of a graphical matroid this definition clearly specialises to the lattice of integer flows of the graph.
The definition of coherent acyclic orientations is somewhat more subtle:

\begin{definition}
Let $\mathcal{N}$ be a submatroid of a regular matroid $\calM$. Let $E(\mathcal{N})$ be the ground set of $\mathcal{N}$. An orientation $D$ of $\mathcal{N}$ is coherent acyclic if for every $e \in E(\mathcal{N})$, there exists $z_e \in (\mathbb{Z}_{\geq 0})^{E(\mathcal{N})}$ such that $e + z_e \in \Gamma(\calM)$, where the columns of $M$ corresponding to $E(\mathcal{N})$ are oriented according to $D$.
\end{definition}

We prove that this definition, when specialised to graphical matroids, is equivalent to that of \cite{Amini}.
\begin{proposition}
Let $H$ be a cut subgraph of a directed graph $G$. Let $\mathcal{M}=\mathcal{M}(G)$ be the corresponding oriented graphical matroid, which is regular and represented by the signed incidence matrix $M$. Let $\mathcal{N}$ be the corresponding submatroid with orientation induced by $G$. If this is a coherent acyclic orientation for $H$, then so is the orientation of $\mathcal N$ induced from $M$.
\end{proposition}
\begin{proof}
Fix an edge $e=(u,v)\in E$. Since $H$ is a cut subgraph, there exists an ordered partition $V = V_1 \sqcup \cdots \sqcup V_s$ such that
$u \in V_i$ and $v \in V_j$, with $i \leq j$. 
To prove the statement, we need to exhibit $z_{e}\in (\mathbb{Z}_{\geq 0})^{E(\mathcal{N})}$ such that $e+z_{e}\in \Gamma(M)$. 

Define
$U_1 = \bigcup_{k \leq i} V_{k}$ and  $U_2 = \bigcup_{k > i} V_k$. Then, $V(G) = U_1 \sqcup U_2$ is an ordered partition such that $e$ is a cut edge. In fact, all edges of $G$ which run between $U_1$ and $U_2$ are edges of $H$, and hence are all oriented from $U_1$ to $U_2$.

Recall that the row vector $r_v$ of $M$ corresponding to a vertex $v \in V(G)$ is a vector in $\mathbb{R}^E$ such that, for any $f\in E$,
$$\inner{r_v, f} =  \begin{cases}
\phantom{-}0 & \text{if $f$ is a loop, or $f$ is not incident to $v$}\\
-1 & \text{if $f$ begins at $v$ and not a loop,}\\
+1 & \text{if $f$ ends at $v$ and not a loop.}\\
\end{cases}$$
Consider $r_2 := \sum_{v \in U_2} r_v$, and note that $r_2 \in \Gamma(G)$ by definition. Note that $\inner{r_2,f}=0$ whenever $f$ is an edge within $U_2$ or $U_1$. Since $e$ is an edge between $U_1$ and $U_2$, it is oriented towards $U_2$ and $\inner{r_2,e}=1$. Therefore, $z_{e}:=r_2-e$ has the required properties. 
\end{proof}

For the reverse implication, we make use of the following Lemma: 

\begin{lemma}\label{lem:Coefficients} Let $G$ be a directed graph, $\mathcal M(G)$ the graphical matroid represented by the signed incidence matrix $M$, and $\mathcal N$ a submatroid such that the orientation of $\mathcal N$ inherited from $M$ is coherent acyclic. 

Let $z_{e} \in (\mathbb{Z}_{\geq 0})^{E(\mathcal{N})}$ such that $c_{e}=e+z_{e} \in \Gamma(\mathcal M)$, and write $c_{e_0} = \sum_{w \in V} a_w r_w$, where $r_w$ denotes the row of $M$ corresponding to $w \in V(G)$, and $a_w\in \Z$. Then, the coefficients $a_w$ have the following properties:
\begin{enumerate}
\item For $e=(u,v) \not \in \mathcal{N}$, $a_u = a_v$.
\item For $e=(u,v) \in \mathcal{N}$, $a_u < a_v$.
\end{enumerate}
\end{lemma}
\begin{proof}
By definition of the signed incidence matrix, we have
$$\inner{c_e, e} = a_v\inner{r_v, e}  + a_u \inner{r_u, e} = a_v - a_u$$
For the first claim, since $e \not \in \mathcal{N}$, we have $\inner{c_e, e} = 0$, and thus, $a_v = a_u$.
For the second claim, since $e \in \mathcal{N}$, we have $\inner{c_e, e} > 0$ since $c_e=e+z_e$ and $z_e$ is a non-negative edge vector. Hence, $a_v > a_u$.
\end{proof}

\begin{proposition}
Let $H$ be a subgraph of a directed graph $G$. Let $\mathcal{M}=\mathcal{M}(G)$ be the graphical matroid represented by the signed incidence matrix $M$. Let $\mathcal{N}$ be the submatroid corresponding to $H$. 

If the orientation of $\mathcal{N}$ induced by $\mathcal{M}(G)$ is coherent acyclic, then $H$ is a cut subgraph of $G$ and its orientation induced by $G$ is coherent acyclic.
\end{proposition}
\begin{proof}
Let $V_1, V_2, \cdots, V_s$ be connected components of $G \setminus H$. 
Assume that the orientation of $\mathcal{N}$ induced by $\mathcal{M}(G)$ is coherent acyclic. That is, for any $e \in H$, there exists $z_{e}\in (\mathbb{Z}_{\geq 0})^{E(\mathcal{N})}$ with $c_{e}=e+z_{e}\in \Gamma(\mathcal{M})$. For $w\in V(G)$, let the coefficients $a_w$ be defined as in Lemma~\ref{lem:Coefficients}.

By Lemma~\ref{lem:Coefficients}, if $u, v \in V_i$, then $a_u = a_v$. Thus we define $a_i := a_v$ for some (hence all) $v \in V_i$, for $i = 1, 2, \cdots, s$. Permute the partition if necessary, to ensure that $a_1 \leq a_2 \leq \cdots \leq a_s$, and in fact, by part (2) of Lemma~\ref{lem:Coefficients}, we know that $a_1 < a_2 < \cdots < a_s$.

First, we prove that $H$ is a cut subgraph with respect to the partition $V(G)=V_1\sqcup V_2 \sqcup ... \sqcup V_s$. By definition of the partition, for any edge $e=(u,v)$ with $e\in G\setminus H$, the two endpoints of $e$ are in the same partition set: $V_i$ for some $1\leq i\leq s$.
For an edge $e=(u,v)$ with $e \in H$, by Lemma~\ref{lem:Coefficients},  $a_u < a_v$, hence, $u$ and $v$ cannot be in the same partition set.

Finally, since $a_1<a_2<\cdots<a_s$, the orientation on $H$ induced from $G$ is coherent acyclic with respect to this partition.
\end{proof}

Using these definitions, the posets $\mathcal{FP}^c$ and $\mathcal{CAC}$ are defined for regular matroids just like they are for graphs.

\subsection{Duality for regular matroids and Amini's theorem for cuts}
We begin with a brief recall of the definition of matroid duality for regular matroids.
Let $\calM$ be a regular matroid represented by a full rank TU matrix $M$, and let $B \subseteq E$ be a basis set of $\calM$, that is, a maximal independent set. Then there is a signed permutation matrix bringing the labels in $B$ to the first $r$ positions, and a matrix $F$, invertible over $\mathbb Z$, such that $FMP=[I_r |  L]$\footnote{For an $r\times s$ matrix $A$ and $r\times t$ matrix $B$, the notation $[A |  B]$ stands for the $r\times (s+t)$ matrix obtained by writing $A$ and $B$ side by side.}, for some $r\times s$ matrix $L$ and the identity matrix $I_r$. This is called a {\em representation of $\calM$ coordinatised by B}, and $[I_r  | L]$ is TU. For more detail see \cite[Section 2.2]{Su-Wagner} and \cite[Chapter 2.2]{Oxley}.

The dual matroid $\calM^\vee$ of $\calM$ is the regular matroid represented by the TU matrix $M^\vee=[-L^T |  I_s],$ where the superscript $T$ denotes the matrix transpose. Note that there is a canonical identification $E(M)=E(M^\vee)$, and the basis sets of $\calM^\vee$ are the complements of the basis sets of $\calM$. In fact, duality is sometimes defined by this property. 

Matroid duality generalises planar graph duality: if $G$ is a planar graph and $G^\vee$ its planar dual, then the corresponding graphical matroids are matroid duals: $\mathcal{M}(G^\vee)=(\mathcal{M}(G))^\vee$. Furthermore, graphical matroids whose matroid dual is also graphical are precisely those obtained from planar graphs.

Observe that $\Gamma(\mathcal{M^\vee}) = \Row(M^\vee) = \ker(M) = \Lambda(\calM)$ canonically, and vice versa, $\Lambda(\mathcal{M}^\vee)=\Gamma(\mathcal{M})$.

The following theorem generalises Amini's theorem for cuts to regular matroids:

\begin{theorem}
For a regular matroid $\mathcal{M}$, the posets $\mathcal{FP}^c(\mathcal{M})$ and $\mathcal{CAC}(\mathcal{M})$ are isomorphic.
\end{theorem}

\begin{proof}
This is Theorem~\ref{main} stated for the dual matroid. Namely, since $\Gamma(\mathcal M)\cong \Lambda(\mathcal M^\vee)$, we have $\mathcal{FP}^c(\mathcal M)\cong \mathcal{FP}(\mathcal{M}^\vee)$. 
Now assume that $\mathcal{N}$ is a submatroid of $\mathcal M$ with a coherent acyclic orientation. That is, for any $e\in E(\mathcal N)$ there exists a non-negative edge vector  $z_e$ with $e+z_e\in \Gamma(\mathcal M)=\Lambda(\mathcal M^\vee)$. In other words, any $e \in E(\mathcal N^\vee)$ is part of a flow,  so $\mathcal N^\vee$ is a submatroid of $\mathcal M^\vee$ with a strongly connected orientation. Thus, $$\mathcal{CAC}(\mathcal M)\cong \mathcal{SC}(\mathcal M^\vee).$$ 
\end{proof}

\end{document}